\documentclass[reqno,11pt]{amsart}

\usepackage{mathdots}
\usepackage{tikz}
\usepackage{tikz-cd}
\usepackage{amssymb}
\usepackage{amsgen}
\usepackage{amsmath}
\usepackage{amsthm}
\usepackage{mathrsfs}
\usepackage{cite}
\usepackage{amsfonts}
\usepackage{enumitem}

\newcommand{\inv}{^{-1}}

\newcommand{\wh}{\widehat}

\usepackage{xcolor}

\newtheorem{Thm}{Theorem}[section]
\newtheorem{Prop}[Thm]{Proposition}

{\theoremstyle{definition}
}
{\theoremstyle{remark}
}

{\theoremstyle{remark}
}
{\theoremstyle{remark}
}

{\theoremstyle{remark}
}
{\theoremstyle{remark}
}
{\theoremstyle{remark}
}
{\theoremstyle{remark}
\newtheorem*{Claim*}{Claim}}

\numberwithin{equation}{section}

\title[Contractibility of the orbit space of the $p$-subgroup complex]{Contractibility of the orbit space of the $p$-subgroup complex via Brown-Forman discrete Morse theory}

\author{Benjamin Steinberg}
\address[B.~Steinberg]{%
    Department of Mathematics\\
    City College of New York\\
    Convent Avenue at 138th Street\\
    New York, New York 10031\\
    USA}
\email{bsteinberg@ccny.cuny.edu}

\thanks{The author was supported by Simons Foundation Collaboration Grant, award number 849561.}
\date{\today}

\begin{document}
\maketitle

\begin{abstract}
We give a simple proof that the orbit space of the $p$-subgroup complex of a finite group is contractible using Brown-Forman discrete Morse theory.  This result was originally conjectured by Webb and proved by Symonds.
\end{abstract}

\section{Introduction}
The order complex  $|\mathcal S_p(G)|$  of the poset $\mathcal S_p(G)$ of nontrivial $p$-subgroups of a finite group $G$ was introduced by Brown in~\cite{BrownSylow}.  The subject was then further developed and popularized by Quillen~\cite{Quillenposet}, and studied by many others, cf.~\cite{Webbsubgroupcomplex,subgroupcomplexesbook}.

There is a natural action of $G$ on $|\mathcal S_p(G)|$ by simplicial maps induced by the conjugation action on subgroups.  Webb conjectured~\cite{Webbsubgroupcomplex} that $|\mathcal S_p(G)|/G$ is contractible (he proved it is $\mathbb F_p$-acyclic).  The conjecture was proved by P.~Symonds~\cite{Symonds}.  Other proofs were given in~\cite{Bux,Linckelmann,Libman}.

Here we provide a simple proof that $|\mathcal S_p(G)|/G$ is contractible using Brown-Forman discrete Morse theory~\cite{brownrewriting,Forman}.  Like the proofs of Symonds~\cite{Symonds} and Bux~\cite{Bux}, we use the Robinson $p$-subgroup complex $\mathcal R_p(G)$, which is $G$-homotopy equivalent to $|\mathcal S_p(G)|$~\cite{ThevWebbGposet}.   The proof of Bux~\cite{Bux} uses discrete Morse theory in the sense of Bestvina-Brady~\cite{BestvinaBrady}, which at least to this author seems more involved than Brown-Forman discrete Morse theory.  Moreover, Bux's  proof is by minimal counterexample, and hence is of an inductive nature; this would appear to obscure the exact sequence of elementary collapses one needs to perform.  We construct a Morse matching~\cite{Chari} for a simplicial set, whose geometric realization is $\mathcal R_p(G)/G$, that has a single critical cell: the vertex corresponding to the conjugacy class of $p$-Sylow subgroups of $G$.  Since the geometric realization is a finite cell complex, the matching provides a recipe of how to perform a sequence of elementary collapses of the complex to a single vertex.     The only group theory used in the proof is the Sylow theorems and the fact that every subgroup of a finite $p$-group is subnormal.

\section{Discrete Morse theory}
The discrete Morse theory we use here was invented independently by Brown~\cite{brownrewriting} and Forman~\cite{Forman}.  We shall mostly follow Brown's approach, as he works with simplicial sets, which is our context as well.  But our terminology will to some extent match that of Forman's theory, and we use the language of Morse matchings~\cite{Chari}.

Let $X$ be a simplicial set.  A \emph{Morse matching} (collapsing scheme in~\cite{brownrewriting}) for $X$ consists of the following data:
\begin{enumerate}
\item A partition of the nondegenerate simplices (called cells here) of $X$ into three types: \emph{critical} (dubbed essential in~\cite{brownrewriting}), \emph{collapsible} and \emph{redundant}.
\item A bijection $c$ between the sets of redundant cells and collapsible cells such that, for each redundant $n$-cell $\tau$,  we have that $c(\tau)\in X_{n+1}$  and a distinguished index $\iota(\tau)\in \{0,\ldots, n+1\}$ such that $\tau = d_{\iota(\tau)}(c(\tau))$.
\end{enumerate}
The crucial condition required of a Morse matching is that if $D$ is the digraph whose vertex set consists of the redundant and collapsible cells and whose edges are of the form $\tau\to c(\tau)$ for $\tau$ redundant and $\sigma\to d_j(\sigma)$, $j\neq \iota(c\inv(\sigma))$, for $\sigma$ collapsible, then there is no infinite directed path in $D$.

%A couple of remarks are in order.

To make our formulation closer to that of~\cite{brownrewriting}, let us call a directed path in $D$ \emph{alternating} if it starts at a redundant cell, and each edge either goes from a redundant cell to a collapsible cell or from a collapsible cell to a redundant cell.    We observe that $D$ has an infinite directed path if and only if it has an infinite alternating path.  To see this, note that any infinite directed path must use an edge from a redundant cell to a collapsible cell.  Indeed, there are no edges from a redundant cell to a redundant cell.  All edges between collapsible cells go from a cell of dimension $n$ to a cell of dimension $n-1$, and so no infinite directed path contains only edges of this type.  In particular, if there is an infinite path in $D$, then there is an infinite path starting at a redundant cell and that path must redundant cells infinitely often.

 Next we claim that if $\sigma$ and $\tau$ are redundant cells and there is a directed path from $\sigma$ to $\tau$, then $\dim \tau\leq \dim \sigma$, and equality holds if and only if the path is alternating.  It suffices by induction on the number of redundant cells in the path to consider the case that there are no redundant cells on this directed path between $\sigma$ and $\tau$.  Then the first edge of this path takes $\sigma$ to $c(\sigma)$, which has dimension one higher.  Note that $c(\sigma)\neq \tau$ as $c(\sigma)$ is collapsible.   Any edge starting at a collapsible cell goes to a cell of lower dimension, and hence all vertices visited on this path after  $c(\sigma)$ have dimension at most that of $\sigma$, and if at least two collapsible cells are visited (i.e., the path is not alternating), then $\dim \tau<\dim \sigma$.  We conclude that the redundant cells visited by an infinite directed path do not increase in dimension and hence they eventually all have the same dimension.  Thus if there is an infinite directed path in $D$, then there is an infinite directed path starting at a redundant cell that visits redundant cells, all of the same dimension, infinitely often.  But such a path must be alternating by the above claim.  In particular, it follows that if $X$ has finitely  many nondegenerate $n$-cells for all $n\geq 0$, then the condition imposed on $D$ is equivalent to acyclicity.  Note that $d_j(c(\tau))\neq \tau$ if $j\neq \iota(\tau)$ for $\tau$ redundant since otherwise there is the directed cycle $\tau\to c(\tau)\to\tau$, which gives rise to an infinite directed path.

We shall write $X_n^r$ for the set of redundant $n$-cells and $X_n^c$ for the set of collapsible $n$-cells.  The following result~\cite[Proposition~1]{brownrewriting} is the fundamental theorem of discrete Morse theory (see also~\cite{Forman,Chari}).

\begin{Thm}[Brown/Forman]\label{t:fund}
Let $X$ be a simplicial set.  Given a Morse matching for $X$, there is a quotient map $q\colon |X|\to Y$ of CW complexes, where the $n$-cells of $Y$ are in bijection with the critical $n$-cells of $X$ (with characteristic maps the compositions of $q$ with the corresponding characteristic maps to $|X|$), with $q$ a homotopy equivalence.
\end{Thm}

The proof uses the hypothesis on $D$ to organize the adjunction of cells so that when a collapsible cell is adjoined, one can perform an elementary collapse of it with the corresponding redundant $n$-cell as the free face.

\section{The orbit space of the $p$-subgroup complex}
Let $G$ be a finite group.  Brown's~\cite{BrownSylow} \emph{$p$-subgroup complex} $|\mathcal S_p(G)|$ is the order complex of the poset  $\mathcal S_p(G)$ of nontrivial $p$-subgroups of $G$ for a prime divisor $p$ of $|G|$.  The group $G$ acts on $\mathcal S_p(G)$ by order-preserving maps via conjugation and hence acts simplicially of $|\mathcal S_p(G)|$.  We shall consider the $G$-subcomplex of $\mathcal R_p(G)$ whose simplices consists of chains $P_0<\cdots<P_n$ such that $P_i\lhd P_n$ for $0\leq i\leq n$.  This complex was first studied by G.R.~Robinson (cf.~\cite{RobinsonComplex}) and was shown to be $G$-homotopy equivalent to $|\mathcal S_p(G)|$ by Th\'evanez and Webb~\cite[Theorem~2]{ThevWebbGposet}.  It follows that $|\mathcal S_p(G)|/G$ is homotopy equivalent to $\mathcal R_p(G)/G$.

We first define the simplicial set $R_p(G)$ whose $n$-simplices are all $n$-tuples $(P_0,\ldots, P_n)$ with $P_0\leq\cdots\leq P_n$ and $P_i\lhd P_n$ for all $0\leq i\leq n$.  The face maps are given by $d_i(P_0,\ldots, P_n) = (P_0,\ldots, \wh P_i,\ldots, P_n)$ (where $\wh P_i$ means omit $P_i$) and the degeneracies are given by $\sigma_i(P_0,\ldots, P_n) = (P_0,\ldots, P_i,P_i,\ldots, P_n)$.  The nondegenerate simplices are those of the form $(P_0,\ldots, P_n)$ with $P_0<\cdots<P_n$ and the geometric realization of $R_p(G)$ is $\mathcal R_p(G)$. This is the standard simplicial set associated to an ordered simplicial complex.

 The group $G$ acts on $R_p(G)$ by simplicial automorphisms via the action $g(P_0,\ldots, P_n) = (gP_0g\inv,\ldots, gP_ng\inv)$.  Therefore, we can form the quotient simplicial set $R_p(G)/G$.  Since geometric realization preserves colimits, we have that $|R_p(G)/G|\cong \mathcal R_p(G)/G$.  We write $[P_0,\ldots, P_n]$ for the orbit of $(P_0,\ldots, P_n)$ and we note that $[P_0,\ldots, P_n]$ is nondegenerate if and only if $P_0<\cdots <P_n$, and hence any face of a nondegenerate simplex is nondegenerate.  Moreover, no two faces of a simplex are identified in the geometric realization. However, the geometric realization need not be a simplicial complex because the cells are not determined by their vertices. %One can have $P<Q$ and $P'<Q'$ where $P,P'$ and $Q,Q'$ are separately conjugate but not simultaneously conjugate.

Our goal is to define a Morse matching on $X:=R_p(G)/G$.  The Sylow theorems guarantee that there is only one orbit of $p$-Sylow subgroups.  We take as the unique  \emph{critical cell} the vertex $[P]$ with $P$ a $p$-Sylow subgroup of $G$.  If $P_0\leq \cdots \leq P_n$, put $N_G(P_0,\ldots, P_n) = \bigcap_{i=0}^nN_G(P_i)$.  We say that $[P_0,\ldots, P_n]$ is \emph{redundant} if $P_n$ is not a $p$-Sylow subgroup of $N_G(P_0,\ldots, P_n)$.  If $n\geq 1$, we say that $[P_0,\ldots, P_n]$ is \emph{collapsible} if $P_n$ is a $p$-Sylow subgroup of $N_G(P_0,\ldots, P_n)$.  One easily checks that being redundant or collapsible is independent of the choice of orbit representative. % since $N(gP_0g\inv,\ldots, gP_ng\inv) = gN(P_0,\ldots, P_n)g\inv$.

It is clear that the collection of critical, redundant and collapsible cells is pairwise disjoint and that any cell of dimension $n\geq 1$ is either redundant or collapsible.  If $P_0$ is a $p$-subgroup that is not $p$-Sylow, then there is a $p$-Sylow subgroup $P$ with $P_0<P$.  Then since any subgroup of a $p$-group is subnormal, it follows that $P_0<N_P(P_0)\leq N_G(P_0)$ and so $[P_0]$ is redundant. %, as $N_P(P_0)$ is a larger $p$-subgroup of $N_G(P_0)$ than $P_0$.
To finish our construction of a Morse matching, it remains to define $c$ and $\iota$ and check that the associated digraph $D$ has no infinite alternating path.

If $[P_0,\ldots, P_n]$ is redundant, then we can find by the Sylow theorems a $p$-Sylow subgroup $P$ of $N_G(P_0,\ldots, P_n)$ with $P_n<P$.  Put $c([P_0,\ldots, P_n]) =[P_0,\ldots, P_n,P]$ and $\iota([P_0,\ldots, P_n]) = n+1$.   Note that $P_0,\ldots, P_n\lhd P$.% by definition of $N_G(P_0,\ldots, P_n)$. % We must check that $c$ meets the criteria to provide a Morse matching.

\begin{Prop}\label{p:c.well.def}
The mapping $c$ is a well-defined bijection between redundant cells and collapsible cells.
\end{Prop}
\begin{proof}
Suppose that $[P_0,\ldots, P_n]$ is redundant and $P$ is a $p$-Sylow subgroup of $N_G(P_0,\ldots, P_n)$ containing $P_n$.
First observe that if $Q$ is any other $p$-Sylow of $N_G(P_0,\ldots,P_n)$ containing $P_n$, then there is $g\in N_G(P_0,\ldots, P_n)$ with $gQg\inv=P$ by the Sylow theorems.  Then $g(P_0,\ldots, P_n,Q)=(P_0,\ldots, P_n,P)$, and so $[P_0,\ldots, P_n,Q]=[P_0,\ldots, P_n,P]$.  Now suppose  $h(P_0',\ldots, P_n')=(P_0,\ldots, P_n)$.  Then $N_G(P_0,\ldots, P_n) = hN_G(P_0',\ldots, P_n')h\inv$, and so if $Q$ is a $p$-Sylow subgroup of $N_G(P_0',\ldots, P_n')$ containing $P_n'$, then $hQh\inv$ is a $p$-Sylow subgroup of $N_G(P_0,\ldots, P_n)$ containing $P_n$, whence, by the previous argument, $[P_0,\ldots, P_n,P]= [P_0,\ldots, P_n,hQh\inv]=[P_0',\ldots, P_n',Q]$.  To complete the proof that $c$ is well defined, we need that $[P_0,\ldots, P_n,P]$ is collapsible.  But since $N_G(P_0,\ldots, P_n,P)\subseteq N_G(P_0,\ldots, P_n)$ and $P$ is a $p$-Sylow subgroup of the latter group, we conclude that $P$ is $p$-Sylow in the former, and hence $[P_0,\ldots, P_n,P]$ is collapsible. Note $c$ is injective as $\tau=d_{n+1}(c(\tau))$, for $\tau \in X_n^r$.

It remains to show that every collapsible $(n+1)$-cell $[P_0,\ldots, P_n,P_{n+1}]$ with $n\geq 0$ is in the image of $c$.   First we claim that $P_{n+1}$ is a $p$-Sylow subgroup of $N_G(P_0,\ldots, P_n)$.  If this is not the case, then by the Sylow theorems, there is a $p$-Sylow subgroup $Q$ of $N_G(P_0,\ldots, P_n)$ with $P_{n+1}<Q$.  Since $P_{n+1}$ is subnormal in $Q$, we must have $P_{n+1}<N_Q(P_{n+1})\subseteq Q\cap N_G(P_{n+1})\subseteq N_G(P_0,\ldots, P_{n+1})$, contradicting that $P_{n+1}$ is a $p$-Sylow subgroup of $N_G(P_0,\ldots, P_{n+1})$.  It now follows that $[P_0,\ldots, P_n]$ is redundant and $c([P_0,\ldots, P_n]) = [P_0,\ldots, P_{n+1}]$, completing the proof.
\end{proof}

All that remains is to show that $D$ has no infinite alternating path.  Define $h([P_0,\ldots, P_n]) = \log_p |P_n|$.  Clearly, $h$ is well defined and bounded by the $p$-adic valuation $\nu_p(|G|)=:t$.  Note that if $\tau$ is a redundant cell, then $h(\tau)<h(c(\tau))$ by construction.  On the other hand, if $\sigma$ is an $(n+1)$-cell, then $h(d_i(\sigma))=h(\sigma)$ for any $0\leq i<n+1$.   Since $\iota(\tau)=n+1$ for any redundant $n$-cell $\tau$,  if $\sigma=[P_0,\ldots, P_{n+1}]$ is collapsible, then $h(\sigma)=h(\sigma')$ whenever there is an edge $\sigma\to \sigma'$ in $D$.  It follows that  no alternating path in $D$ between redundant cells has length greater than $2(t-1)$, and hence $D$ has no infinite alternating path.   Thus our Morse matching collapses $\mathcal R_p(G)/G$ to the vertex corresponding to the conjugacy class of $p$-Sylow subgroups by Theorem~\ref{t:fund}.  Hence we have proved the following.

\begin{Thm}[Symonds~\cite{Symonds}]
If $G$ is a finite group, then $|\mathcal S_p(G)|/G$ is contractible.
\end{Thm}

%\bibliographystyle{abbrv}
%\bibliography{standard2}

\begin{thebibliography}{10}

\bibitem{BestvinaBrady}
M.~Bestvina and N.~Brady.
\newblock Morse theory and finiteness properties of groups.
\newblock {\em Invent. Math.}, 129(3):445--470, 1997.

\bibitem{BrownSylow}
K.~S. Brown.
\newblock Euler characteristics of groups: the {$p$}-fractional part.
\newblock {\em Invent. Math.}, 29(1):1--5, 1975.

\bibitem{brownrewriting}
K.~S. Brown.
\newblock The geometry of rewriting systems: a proof of the
  {A}nick-{G}roves-{S}quier theorem.
\newblock In {\em Algorithms and classification in combinatorial group theory
  ({B}erkeley, {CA}, 1989)}, volume~23 of {\em Math. Sci. Res. Inst. Publ.},
  pages 137--163. Springer, New York, 1992.

\bibitem{Bux}
K.-U. Bux.
\newblock Orbit spaces of subgroup complexes, {M}orse theory, and a new proof
  of a conjecture of {W}ebb.
\newblock In {\em Proceedings of the 1999 {T}opology and {D}ynamics
  {C}onference ({S}alt {L}ake {C}ity, {UT})}, volume~24, pages 39--51, 1999.

\bibitem{Chari}
M.~K. Chari.
\newblock On discrete {M}orse functions and combinatorial decompositions.
\newblock volume 217, pages 101--113. 2000.
\newblock Formal power series and algebraic combinatorics (Vienna, 1997).

\bibitem{Forman}
R.~Forman.
\newblock Morse theory for cell complexes.
\newblock {\em Adv. Math.}, 134(1):90--145, 1998.

\bibitem{RobinsonComplex}
R.~Kn\"{o}rr and G.~R. Robinson.
\newblock Some remarks on a conjecture of {A}lperin.
\newblock {\em J. London Math. Soc. (2)}, 39(1):48--60, 1989.

\bibitem{Libman}
A.~Libman.
\newblock Webb's conjecture for fusion systems.
\newblock {\em Israel J. Math.}, 167:141--154, 2008.

\bibitem{Linckelmann}
M.~Linckelmann.
\newblock The orbit space of a fusion system is contractible.
\newblock {\em Proc. Lond. Math. Soc. (3)}, 98(1):191--216, 2009.

\bibitem{Quillenposet}
D.~Quillen.
\newblock Homotopy properties of the poset of nontrivial {$p$}-subgroups of a
  group.
\newblock {\em Adv. in Math.}, 28(2):101--128, 1978.

\bibitem{subgroupcomplexesbook}
S.~D. Smith.
\newblock {\em Subgroup complexes}, volume 179 of {\em Mathematical Surveys and
  Monographs}.
\newblock American Mathematical Society, Providence, RI, 2011.

\bibitem{Symonds}
P.~Symonds.
\newblock The orbit space of the {$p$}-subgroup complex is contractible.
\newblock {\em Comment. Math. Helv.}, 73(3):400--405, 1998.

\bibitem{ThevWebbGposet}
J.~Th\'{e}venaz and P.~J. Webb.
\newblock Homotopy equivalence of posets with a group action.
\newblock {\em J. Combin. Theory Ser. A}, 56(2):173--181, 1991.

\bibitem{Webbsubgroupcomplex}
P.~J. Webb.
\newblock Subgroup complexes.
\newblock In {\em The {A}rcata {C}onference on {R}epresentations of {F}inite
  {G}roups ({A}rcata, {C}alif., 1986)}, volume~47 of {\em Proc. Sympos. Pure
  Math.}, pages 349--365. Amer. Math. Soc., Providence, RI, 1987.

\end{thebibliography}
\def\malce{\mathbin{\hbox{$\bigcirc$\rlap{\kern-7.75pt\raise0,50pt\hbox{${\tt
  m}$}}}}}\def\cprime{$'$} \def\cprime{$'$} \def\cprime{$'$} \def\cprime{$'$}
  \def\cprime{$'$} \def\cprime{$'$} \def\cprime{$'$} \def\cprime{$'$}
  \def\cprime{$'$} \def\cprime{$'$}

\end{document}